\documentclass[12pt,a4paper]{article}

\usepackage{latexsym}
\usepackage{amsmath}
\usepackage{amssymb}
\usepackage{arydshln}

\usepackage{cite}
\usepackage{stmaryrd}
\usepackage{enumerate}

\usepackage{hyperref}

\usepackage{color}
\usepackage{lineno}
\usepackage{graphicx}
\usepackage{ae}
\usepackage{amsmath}
\usepackage{amssymb}
\usepackage{latexsym}
\usepackage{url}
\usepackage{epsfig}
\usepackage{mathrsfs}
\usepackage{amsfonts}
\usepackage{amsthm}

\usepackage{tipa} 

\usepackage{float}
\usepackage{subfig}
\usepackage{tikz}
\usetikzlibrary{positioning}

\newcommand{\trace}{\mathrm{trace}}

\newcommand{\diag}{\mathrm{diag}}

\newtheorem{theorem}{Theorem}[section]

\newtheorem{lemma}[theorem]{Lemma}
\newtheorem{remark}{Remark}
\newtheorem{cor}[theorem]{Corollary}

\theoremstyle{definition}

\newtheorem{claim}{Claim}

\newtheorem{conjecture}{Conjecture}

\numberwithin{equation}{section} 
\allowdisplaybreaks    

\def\qed{\hfill$\Box$\vspace{12pt}}

\long\def\delete#1{}

\usetikzlibrary{decorations.markings}

\tikzstyle{vertex}=[circle, draw, inner sep=0pt, minimum size=6pt]
\tikzstyle{directed}=[postaction={decorate,
	decoration={markings,mark=at position 0.3 with {\arrow{stealth}}}
}]

\usepackage{xcolor}
\usepackage[normalem]{ulem}

\begin{document}
\title {Bounds for the largest eigenvalue and sum of Laplacian eigenvalues of signed graphs}

\author{Linfeng Xie$^{a,b}$,~Xiaogang Liu$^{a,b,c,}$\thanks{Supported by the National Natural Science Foundation of China (No. 12371358) and the Guangdong Basic and Applied
		Basic Research Foundation (No. 2023A1515010986).}~$^,$\thanks{ Corresponding author. Email addresses: xielinfeng@mail.nwpu.edu.cn, xiaogliu@nwpu.edu.cn}~
	\\[2mm]
	{\small $^a$School of Mathematics and Statistics,}\\[-0.8ex]
	{\small Northwestern Polytechnical University, Xi'an, Shaanxi 710072, P.R.~China}\\
	{\small $^b$Xi'an-Budapest Joint Research Center for Combinatorics,}\\[-0.8ex]
	{\small Northwestern Polytechnical University, Xi'an, Shaanxi 710129, P.R. China}\\
	{\small $^c$Research \& Development Institute of Northwestern Polytechnical University in Shenzhen,}\\[-0.8ex]
	{\small Shenzhen, Guandong 518063, P.R. China}\\
}
\date{}

\openup 0.5\jot
\maketitle

\begin{abstract}
	In this paper, we consider the bounds for the largest eigenvalue and the sum of the $k$ largest Laplacian eigenvalues of signed graphs. Firstly, we give an upper bound on the largest eigenvalue of the adjacency matrix of a signed graph and characterize the extremal graphs that attain this bound. Secondly, we prove that a non-bipartite signed graph $\Gamma$ of order $n$ and size $m$ contains a balanced triangle if $\lambda_{1}(\Gamma)\ge \sqrt{m-1}$, $\lambda_{1}(\Gamma) \ge |\lambda_{n}(\Gamma)|$ and $\Gamma\not \sim (C_{5}\cup (n-5)K_{1},+)$, where $\lambda_{1}(\Gamma)$ is the largest eigenvalue of the adjacency matrix of $\Gamma$. Thirdly, we confirm a conjecture proposed in [Linear Multilinear Algebra 51 (1) (2003) 21--30] that: if $\Gamma$ is  a connected signed graph, then
	$$
	\sum_{i=1}^{k}\mu_{i}(\Gamma) >\sum_{i=1}^{k}d_{i}(\Gamma)~~(1\le k\le n-1),
	$$
	where $\mu_{1}(\Gamma)\ge\mu_{2}(\Gamma)\ge\cdots \ge \mu_{n}(\Gamma)$ are Laplacian eigenvalues of $\Gamma$, and $d_{1}(\Gamma)\ge d_{2}(\Gamma)\ge \dots \ge d_{n}(\Gamma)$ are vertex degrees of $\Gamma$. Finally, we give a lower bound for the sum of the $k$ largest Laplacian eigenvalues of a connected signed graph.
	
	\smallskip
	
	
\end{abstract}

\section{Introduction}
A \emph{signed graph} $\Gamma=(G,\sigma)$ consists of an unsigned graph $G=(V(G),E(G))$ and a sign function $\sigma : E  \to \left \{+1, -1 \right \}$, where $G$ is called the \emph{underlying graph} of $\Gamma$. In 1946, Heider first introduced the concept of signed graphs in his study of balance theory in social psychology \cite{Hei}. Heider sought to explain how friendly/hostile relationships in social networks influence group stability. In 1953, Harary formalized this theory and established the mathematical foundation of signed graphs \cite{Har}. For more information of signed graphs, please refer to \cite{Cam, Zaslavsky1982, Zaslavsky2008,Zaslavsky2018}.

Let $\Gamma=(G,\sigma)$ be a signed graph with the vertex set $V(\Gamma) =  \left \{v_{1},v_{2},\dots ,v_{n}\right \}$. An edge $uv$ labeled with ``$+1$" (respectively, ``$-1$") is called a \emph{positive edge} (respectively, \emph{negative edge}), denoted by $\sigma(uv)=+1 $ (respectively, $\sigma(uv)=- 1$). All positive edges (respectively, negative edges) in $\Gamma$ are denoted by $E^{+}(\Gamma)$ (respectively, $E^{-}(\Gamma)$). The \emph{adjacency matrix} of $\Gamma$ is denoted by $A(\Gamma)=(a_{ij}^{\sigma})$, where $a_{ij}^{\sigma}=\sigma(v_{i}v_{j})a_{ij}$ and $a_{ij}=1$ if $ v_{i} $ and $v_{j}$ are adjacent, and $a_{ij}=0$ otherwise. The \emph{degree matrix} of $\Gamma$ is the diagonal matrix $D(\Gamma)=\diag(d(v_{1}), d(v_{2}), \dots , d(v_{n}))$, where $d(v_i)$ denotes the number of edges incident to $v_i$. The \emph{Laplacian matrix} of $\Gamma$, denoted by $L(\Gamma)$ or $L(G,\sigma)$, is defined as $D(\Gamma)-A(\Gamma)$. Let  $\lambda_{1}(\Gamma)\ge \lambda_{2}(\Gamma)\ge \dots \ge \lambda_{n}(\Gamma)$ and $\mu_{1}(\Gamma)\ge \mu_{2}(\Gamma)\ge \dots \ge \mu_{n}(\Gamma)$ denote the eigenvalues of $A(\Gamma)$ and $L(\Gamma)$, which are called the \emph{adjacency spectrum} and  \emph{Laplacian spectrum} of $\Gamma$, respectively. Until now, there are lots of results on the spectra of signed graphs \cite{Hou, Bel, Ping, Reff, Liu, Yao, Belardo, And, Wang-Lin}.

Let $\theta : V(\Gamma) \to$ $\left \{+1 , - 1\right \}$ be a \emph{sign function}. \emph{Switching $\Gamma$ by a sign function $\theta$} means forming a new signed graph $\Gamma^{\theta}=(G,\sigma^{\theta})$ whose underlying graph is the same as $G$, and the sign function on an edge $e=v_{i}v_{j}$ is defined by $\sigma^{\theta}(e)=\theta(v_{i})\sigma(e)\theta(v_{j})$. Let $S\subseteq V(\Gamma)$. \emph{Switching $\Gamma$ at a vertex subset $S$} means changing the sign of every edge with exactly one end in $S$. In fact, switching $\Gamma$ by a sign function $\theta$ is equivalent to switching $\Gamma$ at a vertex subset $S=\{v\in V(\Gamma) \mid \theta(v)=-1\}$. Let $\Gamma_{1}=(G,\sigma_{1})$ and $\Gamma_{2}=(G,\sigma_{2})$ be two signed graphs with the same underlying graph $G$. If there exists a sign function $\theta$ such that $\Gamma_{2}=\Gamma^{\theta}_{1}$, then $\Gamma_{1}$ and $\Gamma_{2}$ are \emph{switching~equivalent}, denoted as $\Gamma_{1} \sim \Gamma_{2}$. If not, then $\Gamma_{1} \not\sim \Gamma_{2}$. Here, $\theta$ is also called a \emph{switching~function}. A cycle is called  \emph{positive}  or \emph{balanced} if the number of its negative edges is even; otherwise, \emph{negative}  or \emph{unbalanced}. A signed graph is called \emph{balanced} if each of its cycles is positive; otherwise, \emph{unbalanced}. Two $n\times n$ matrices $M_{1}$ and $M_{2}$ are called \emph{signature similar}, if there exists a diagonal matrix $S=\diag(s_{1},s_{2},\dots,s_{n})$  with diagonal entries $s_{i}=\pm 1$ such that $M_{2}=SM_{1}S^{-1}$. Clearly, two signature similar matrices have the same eigenvalues.

Let $G=(V(G),E(G))$ be a simple graph. Let $\omega(G)$ denote the \emph{clique number} of $G$, which is the order of the largest clique of $G$ and $\lambda_{1}(G)$ the largest adjacency eigenvalue of $G$.  In \cite{Wilf-1986}, Wilf proved that
$$
\lambda_{1}(G)\le \frac{\omega(G)-1}{\omega(G)}n,
$$
where $n$ is the order of $G$.  In \cite{Nik2}, Nikiforov proved that
$$
\lambda_{1}(G)\le \sqrt{2m\frac{\omega(G)-1}{\omega(G)}},
$$
where $m$ is the size of $G$. Let $c(e)$ denote the order of the largest clique containing an edge $e$ in $G$. Recently, Liu and Ning \cite{LB-ar-2023} proved that
\begin{equation}\label{LB-JCTB-2026}
	\lambda_{1}(G)\le \sqrt{2\sum_{e\in E(G)}\frac{c(e)-1}{c(e)}},
\end{equation}
where the equality of \eqref{LB-JCTB-2026} holds if and only if $G$ is a complete bipartite graph for $\omega=2$, or a complete regular $\omega$-partite graph for $\omega\ge 3$ $($possibly with some isolated vertices$)$.

There are analogous results on signed graphs. Let $\omega_{b}(\Gamma)$ denote the \emph{balanced clique number} of a signed graph $\Gamma$, which
is the order of the largest balanced clique of $\Gamma$. In  \cite{Sun}, Sun, Liu and Lan proved that
$$
\lambda_{1}(\Gamma)\le \sqrt{2m\frac{\omega_{b}(\Gamma)-1}{\omega_{b}(\Gamma)}},
$$
where $m$ is the size of a signed graph $\Gamma$. Subsequently, in \cite{Kan}, Kannan and Pragada proved that
\begin{equation}\label{Kan-Th-3.3}
	\lambda_{1}^{2}(\Gamma)\le 2(m-\epsilon(\Gamma))\frac{\omega_{b}(\Gamma)-1}{\omega_{b}(\Gamma)},
\end{equation}
where $\epsilon(\Gamma)$ denotes the \emph{frustration index} of a signed graph $\Gamma$, which is the
minimum number of edges to remove for balance.

In this paper, we firstly give an upper bound on $\lambda_{1}(\Gamma)$, as shown in Theorem \ref{AD1}, which is better than \eqref{Kan-Th-3.3} (See Remark \ref{rem1-1-1}).

\begin{theorem}\label{AD1}
	Let $\Gamma=(G,\sigma)$ be a signed graph of order $n$ with $\omega_{b}(\Gamma)=\omega_{b} \ge 2$. Let $c_{b}(e)$ denote the order of the largest balanced clique containing an edge $e$ in $\Gamma$. Then
	\begin{equation}\label{1}
		\lambda_{1}(\Gamma) \le \sqrt{2\sum_{e\in E^{+}(\Gamma^{\prime})}\frac{c_{b}(e)-1}{c_{b}(e)}},
	\end{equation}
	where  $\Gamma^{\prime} \sim \Gamma$ and $\lambda_{1}(\Gamma^{\prime})$ has a non-negative eigenvector. Equality holds if and only if $\Gamma$ is a balanced complete bipartite graph for $\omega_{b}=2$, or a  balanced complete regular $\omega_{b}$-partite graph for $\omega_{b}\ge 3$ $($possibly with some isolated vertices$)$.
\end{theorem}

Let $G$ be a simple graph of size $m$ and $t(G)$ the number of triangles in $G$.  In \cite[Theorem~2.1]{Nik2}, Nikiforov proved that if $\lambda_{1}(G) \ge \sqrt{m}$, then $t(G)\ge 1$ unless $G$ is a complete bipartite graph. In \cite[Theorem 1.3]{Bo}, Lin, Ning and Wu confirmed that if $G$ is a non-bipartite graph of size $m$ and $\lambda_{1}(G) \ge \sqrt{m-1}$, then $t(G)\ge 1$  unless $G$ is $C_{5}$ (possibly with isolated vertices). Let $SK_{s,t}$ denote the graph obtained by subdividing an edge of the complete bipartite graph $K_{s,t}$. In \cite[Theorem 1.4]{Bo}, Lin, Ning and Wu proved that if $G$ is a non-bipartite graph of order $n$ and $\lambda_{1}(G)\ge \lambda_1(SK_{\lfloor \frac{n-1}{2} \rfloor,\lceil \frac{n-1}{2}\rceil})$, then $t(G)\ge 1$ unless $G$ is $SK_{\lfloor \frac{n-1}{2} \rfloor,\lceil \frac{n-1}{2}\rceil}$. In \cite[Theorem 1.4]{Zhai}, Zhai and Shu proved that if $G$ is a non-bipartite graph of size $m$ and $\lambda_{1}(G)\ge \lambda_1(SK_{2,\frac{m-1}{2}})$, then $t(G)\ge 1$ unless $G$ is $SK_{2,\frac{m-1}{2}}$.  For more results on this research, see \cite{Li-Feng-Peng, Wang22, Bo-Zhai}.

Let $(G, +)$ denote a signed graph with all positive edges. The signed graph $\Gamma^{1,n-3}$ is obtained from a signed clique $(K_{n-2},+)$ with $V(K_{n-2})=\left \{u_{1},v_{1},\dots,v_{n-3} \right \}$ and two isolated vertices $u$ and $v$ by adding a negative edge $uv$ and positive edges $\left \{uu_{1},vv_{1},\dots,vv_{n-3}\right \}$. In \cite{Wang-Hou-Li}, Wang, Hou and Li proved that  if $\Gamma$ is a connected unbalanced signed graph of order $n$ and $\Gamma$ is $C_3^-$-free, then
$$
\rho(\Gamma)\le \frac{1}{2}\sqrt{n^{2}-8}+n-4,
$$
with equality holding if and only if $\Gamma \sim \Gamma^{1,n-3}$, where $C_3^-$ denotes the unbalanced triangle and $\rho(\Gamma)=\max\left \{\lambda_{1}(\Gamma),-\lambda_{n}(\Gamma)  \right \}$.

In this paper, we secondly establish the existence of  balanced triangle based on a lower bound on the largest eigenvalue of a signed graph, as shown in Theorem \ref{AD2}.

\begin{theorem}\label{AD2}
	Let $\Gamma=(G,\sigma)$ be a non-bipartite  signed graph of order $n$ and size $m$. If $\lambda_{1}(\Gamma) \ge \sqrt{m-1}$ and $\lambda_{1}(\Gamma) \ge |\lambda_{n}(\Gamma)|$, then $\Gamma$ contains a balanced triangle unless $\Gamma \sim (C_{5}\cup (n-5)K_{1},+)$.	
\end{theorem}

Let $G$ be a graph of order $n$ with Laplacian eigenvalues $\mu_{1}(G)\ge\mu_{2}(G)\ge\cdots \ge \mu_{n}(G)$ and vertex degrees $d_{1}(G)\ge d_{2}(G)\ge \dots \ge d_{n}(G)$.	In 1994, Grone and Merris proposed the following conjecture.

\begin{conjecture}(See \cite{GM-SIAM-1994})\label{GM-SIAM-1994}
	{ \em Let $G$ be a connected graph of order $n$. Then
		$$
		\sum_{i=1}^{k}\mu_{i}(G)\ge 1+\sum_{i=1}^{k}d_{i}(G)~~(1\le k\le n).
		$$
	}
\end{conjecture}

Subsequently, in 1995, Grone proved the following stronger result, which confirmed Conjecture \ref{GM-SIAM-1994}.

\begin{theorem}\emph{(See \cite{Grone})}\label{Grone}
	Let $G$ be a simple connected graph of order $n$. For each $k=1,2,\dots,n-1$, let $t_{k}$ be the number of components of the graph induced by $G$ on $\left \{1,2,\dots,k  \right \}$. Then
	$$
	\sum_{i=1}^{k}\mu_{i}(G)\ge t_{k}+\sum_{i=1}^{k}d_{i}(G).
	$$
\end{theorem}

However, the above conclusion does not hold for signed graphs. In 2003, Hou, Li and Pan proposed the following conjecture.

\begin{conjecture}(See \cite{Hou})\label{ConjecHou-1}
	{\em Let $\Gamma=(G,\sigma)$ be a connected signed graph of order $n$. Then
		$$
		\sum_{i=1}^{k}\mu_{i}(\Gamma) >\sum_{i=1}^{k}d_{i}(\Gamma)~~(1\le k\le n-1),
		$$
		where $\mu_{1}(\Gamma)\ge\mu_{2}(\Gamma)\ge\cdots \ge \mu_{n}(\Gamma)$ are Laplacian eigenvalues of $\Gamma$, and $d_{1}(\Gamma)\ge d_{2}(\Gamma)\ge \dots \ge d_{n}(\Gamma)$ are  vertex degrees of $\Gamma$.
	}
\end{conjecture}
In \cite{Hou}, Hou, Li and Pan proved that  $\mu_{1}(\Gamma)\ge 1+d_{1}(\Gamma)$, and so Conjecture \ref{ConjecHou-1} is true for $k=1$. In 2015, Liu and Shen\cite{Liu} proved that Conjecture \ref{ConjecHou-1} is also true for $k=2,n-1$.

In this paper, we thirdly confirm Conjecture \ref{ConjecHou-1}.

Let $\Gamma=(G,\sigma)$ be a signed graph. For $U\subseteq V(\Gamma)$, let $N_{U}(v)$ denote the set of neighbors of $v$ in $U$, and $d_{U}(v)=|N_U(v)|$. In particular, if $U=V(\Gamma)$, we write $N_{\Gamma}(v)$ and $d_{\Gamma}(v)$ for simplicity. Let $\Gamma[U]$ denote the signed subgraph of $\Gamma$ induced by $U$, with edge signs inherited from $\Gamma$. Sometimes, we say that $U$ induces $\Gamma[U]$.

Finally, we give another lower bound on the sum of the $k$ largest Laplacian eigenvalues of a signed graph, as shown in Theorem \ref{LP1}.

\begin{theorem}\label{LP1}
	Let $\Gamma=(G,\sigma)$ be a connected signed graph of order $n>2$ and size $m$. Suppose that $R=\left \{u_{1},u_{2},\dots,u_{k}\right \}$ is a subset of the vertex set of $\Gamma$. If $E(\Gamma[R])$ consists of $r$ pairwise disjoint edges, then
	$$
	\sum_{i=1}^{k}\mu_{i}(\Gamma)\ge \sum_{i=1}^{k}d_{\Gamma}(u_{i})+k-r.
	$$
\end{theorem}

We will prove Theorem \ref{AD1} in Section \ref{Sec-Conj-Hou-33}, Theorem \ref{AD2} in Section \ref{Sec-Conj-Hou-44}, Conjecture \ref{ConjecHou-1} in Section \ref{Sec-Conj-Hou-1} and Theorem \ref{LP1} in Section \ref{Sec-Conj-Hou-22}, respectively.


\section{Proof~of~Theorem~\ref{AD1}}\label{Sec-Conj-Hou-33}



\begin{lemma}\emph{(See \cite{Yao})}
	\label{Yao}
	Let $\Gamma_{1}=(G,\sigma_{1})$ and $\Gamma_{2}=(G,\sigma_{2})$ be two signed graphs on the same underlying graph. The followings are equivalent:
	
	$\mathrm{(1)}$~$\Gamma_{1}$ and $\Gamma_{2}$ are switching equivalent.
	
	$\mathrm{(2)}$~$A(\Gamma_{1})$ and $A(\Gamma_{2})$ are signature similar.
	
	$\mathrm{(3)}$~$L(\Gamma_{1})$ and $L(\Gamma_{2})$ are signature similar.
\end{lemma}


\begin{lemma}\emph{(See \cite{Sun})}\label{Sun}
	Let $\Gamma$ be a signed graph with $n$ vertices. Then there exists a signed graph $\Gamma^{\prime}$ switching equivalent to $\Gamma$ such that $\lambda_{1}(\Gamma^{\prime})$ has a non-negative eigenvector.
\end{lemma}

Let $\Gamma=(G,\sigma)$ be a  signed graph with underlying graph $G=(V(G),E(G))$. Let $\emptyset\not=S\subseteq  E(G)$. Denote by $\Gamma-S$ the signed graph obtained from $\Gamma$ by deleting all signed edges of $\Gamma$ whose underlying edges are in $S$.

\begin{lemma}\emph{(See \cite{Lan})}\label{Lan}
	Let $\Gamma$ be an unbalanced signed graph with underlying graph $G=(V(G),E(G))$. Then there exists an edge $e=v_{k}v_{r}\in E(G)$ such that $\lambda_{1}(\Gamma)\le \lambda_{1}(\Gamma-e)$. Furthermore, if $\Gamma-e$ is balanced, then $\lambda_{1}(\Gamma)<\lambda_{1}(\Gamma-e)$.
\end{lemma}

By refining the proof of Lemma \ref{Lan}, we can obtain the following result.

\begin{lemma}\label{LanCor-S}
	Let $\Gamma$ be an unbalanced signed graph with underlying graph $G=(V(G),E(G))$. Then there exists a set $S\subseteq E(G)$ such that $\lambda_{1}(\Gamma)\le \lambda_{1}(\Gamma-S)$. Furthermore, if $\Gamma-S$ is balanced, then $\lambda_{1}(\Gamma)<\lambda_{1}(\Gamma-S)$.
\end{lemma}

\medskip
\begin{Tproof}\textbf{~of~Theorem~\ref{AD1}.}
	Define
	$$
	S^{\pm}_{n}=\left \{(x_{1},x_{2},\dots,x_{n})^T\in \mathbb{R}^{n} : \sum |x_{i}|=1 \right\}.
	$$
	By \cite[Theorem 5]{Wang-Yan-Qian}, we have
	\begin{equation}\label{equation-clique}
		\max\left \{\mathbf{x}^{T}A(\Gamma)\mathbf{x}:  \mathbf{x}\in S^{\pm}_{n} \right \}=1-\frac{1}{\omega_{b}(\Gamma)}.
	\end{equation}
	
	Set $V(\Gamma)=\{1,2, \ldots,n\}$. Let $W_{\Gamma}$ be an $n\times n$ matrix whose $(i,j)$-entry is $\sigma(ij)\frac{c_{b}(e)}{c_{b}(e)-1}$ if $e=ij\in E(\Gamma)$, and 0 otherwise. For $ \mathbf{x}=(x_{1},x_{2},\dots,x_{n})^T\in S^{\pm}_{n}$, define
	\begin{align}\label{SnXBiaoDaF1}
		F_{\Gamma}(\mathbf{x})& :=\mathbf{x}^{T}W_{\Gamma}\mathbf{x} =2\sum_{e=ij\in E(\Gamma)}\sigma(ij)\frac{c_{b}(e)}{c_{b}(e)-1}x_{i}x_{j}.
	\end{align}

	Let $\mathrm{supp}(\mathbf{x})$ denote the \emph{support} of a vector $\mathbf{x}$, which is the set consisting of all indices corresponding to nonzero entries in $\mathbf{x}$.
	
	\begin{claim}\label{Clique1}
		Let $\mathbf{x}\in S^{\pm}_{n}$ be a vector such that 	$F_{\Gamma}(\mathbf{x})  = \max \left \{F_{\Gamma}(\mathbf{z}) : \mathbf{z}\in S^{\pm}_{n} \right \}$ with minimal supporting set. Then $\mathrm{supp}(\mathbf{x})$ induces a balanced clique in $\Gamma$.
	\end{claim}

	\noindent\emph{Proof of Claim 1.}~Let $\mathbf{x}=(x_{1},x_{2},\dots,x_{n})^T\in S^{\pm}_{n}$ such that $\mathbf{x}$ has a minimal supporting set and $F_{\Gamma}(\mathbf{x})  =  \max\left \{F_{\Gamma}(\mathbf{z}) : \mathbf{z}\in S^{\pm}_{n} \right \}$. Define $S=\left \{i : x_{i}<0 \right \}$ and let $\Gamma^{S}$ be the signed graph obtained from $\Gamma$ by a switching at $S$. Let $\mathbf{y}=(|x_{1}|, |x_{2}|, \dots , |x_{n}|)^T\in S^{\pm}_{n}$. By \eqref{SnXBiaoDaF1}, we have $F_{\Gamma^S}(\mathbf{y})=F_{\Gamma}(\mathbf{x})= \max\left \{F_{\Gamma^S}(\mathbf{z}) : \mathbf{z}\in S^{\pm}_{n} \right \}$. We assert that $\mathbf{y}$ is also a vector such that $F_{\Gamma^S}(\mathbf{y})=\max\left \{F_{\Gamma^S}(\mathbf{z}) : \mathbf{z}\in S^{\pm}_{n} \right \}$ with minimal supporting set. Otherwise, suppose that $\mathbf{z}_1=(a_{1},a_{2},\dots,a_{n})^T\in S^{\pm}_{n}$ such that $F_{\Gamma^S}(\mathbf{z}_1)=F_{\Gamma^S}(\mathbf{y})$ and $|\mathrm{supp}(\mathbf{z}_1)| < |\mathrm{supp}(\mathbf{y})|$. Let $\mathbf{z}_2=(b_{1},b_{2},\dots,b_{n})^T$ where $b_i=-a_i$ for $i\in S$ and  $b_i=a_i$ for $i \notin S$. Then
	$$
	F_{\Gamma}(\mathbf{z}_2)=F_{\Gamma^S}(\mathbf{z}_1)=F_{\Gamma^S}(\mathbf{y})=F_{\Gamma}(\mathbf{x}),
	$$
	but $|\mathrm{supp}(\mathbf{z}_2)|=|\mathrm{supp}(\mathbf{z}_1)|<|\mathrm{supp}(\mathbf{y})|=|\mathrm{supp}(\mathbf{x})|$, a contradiction.

	We assert that $\mathrm{supp}(\mathbf{y})$ induces an all-positive complete subgraph in $\Gamma^S$, that is, $\mathrm{supp}(\mathbf{x})$ induces a balanced clique in $\Gamma$. Otherwise, there are two vertices in $\mathrm{supp}(\mathbf{y})$, say $i$ and $j$, such that $ij\notin E(\Gamma^S)$ or $\sigma^S(ij)=-1$, where $\sigma^S$ is a sign function of $\Gamma^S$. Let $\mathbf{e}_{i}$ be the $i$-th column of the identity matrix of order $n$. Without loss of generality, assume that $\mathbf{y}^{T}W_{\Gamma^S}\mathbf{e}_{i}\ge \mathbf{y}^{T}W_{\Gamma^S}\mathbf{e}_{j}$. Set $\mathbf{y}'=\mathbf{y}+|x_{j}|(\mathbf{e}_{i}-\mathbf{e}_{j})$. Clearly, $\mathbf{y}' \in S^{\pm}_{n}$.
	
	If $ij\notin E(\Gamma)$, then
	$$
	F_{\Gamma^S}(\mathbf{y}')-F_{\Gamma^S}(\mathbf{y}) =2|x_{j}|\mathbf{y}^{T}W_{\Gamma^S}(\mathbf{e}_{i}-\mathbf{e}_{j})\ge 0.
	$$
	
	If $\sigma^S(ij)=-1$, then
	$$
	F_{\Gamma^S}(\mathbf{y}')-F_{\Gamma^S}(\mathbf{y}) =2|x_{j}|\mathbf{y}^{T}W_{\Gamma^S}(\mathbf{e}_{i}-\mathbf{e}_{j})+2x_{j}^{2}\frac{c_{b}(e)}{c_{b}(e)-1}>0.
	$$
	
	Note that $|\mathrm{supp}(\mathbf{y}')| < |\mathrm{supp}(\mathbf{y})|$ and $F_{\Gamma^S}(\mathbf{y}')\ge F_{\Gamma^S}(\mathbf{y})$ in either case, contradictions.
	
	
	This completes the proof of Claim \ref{Clique1}.
	
	\begin{claim}\label{Clique2}
		For each $\mathbf{z}\in S^{\pm}_{n}$, $F_{\Gamma}(\mathbf{z})\le 1$.
	\end{claim}
	\noindent\emph{Proof of Claim 2.}~Let $\mathbf{w}=(w_{1},w_{2},\dots,w_{n})^T\in S^{\pm}_{n}$ be a vector satisfying $F_{\Gamma}(\mathbf{w})=\max \left \{F_{\Gamma}(\mathbf{y}) : \mathbf{y}\in S^{\pm}_{n} \right \}$ with minimal supporting set. By Claim \ref{Clique1}, $\mathrm{supp}(\mathbf{w})$ forms a balanced clique $(K,\sigma)$ in $\Gamma$. By \eqref{SnXBiaoDaF1},
	$$
	F_{\Gamma}(\mathbf{w})=2\sum_{e=ij\in E((K,\sigma))}\sigma(ij)\frac{c_{b}(e)}{c_{b}(e)-1}w_{i}w_{j}.
	$$
	Let $r$ be the order of $K$, then $c_{b}(e)\ge r$. Observe that $\frac{x}{x-1}$ is a decreasing function in $x$. Then we have $\frac{c_{b}(e)}{c_{b}(e)-1}\le \frac{r}{r-1}$. Since $(K,\sigma)$ is balanced, there exists a sign function $\theta$ such that $\Gamma^{\theta}\sim \Gamma$ and $(K,\sigma^{\theta})=(K,+)$. Assume that $\theta(i)=-1$ for $i\in A^{-}$ and $\theta(i)=+1$ for $i\in A^+$ with $A^+\cup A^-=V(\Gamma)$. Let
	$$
	\mathbf{w}^{\prime}=\left\{
	\begin{array}{l}
		w^{\prime}_{i}=-w_{i}, ~~i\in A^-, \\[0.2cm]
		w^{\prime}_{i}=w_{i}, ~~i\in A^+.
	\end{array}\right.
	$$
	By  \eqref{equation-clique} and \eqref{SnXBiaoDaF1}, we obtain
	\begin{align*}
		F_{\Gamma}(\mathbf{w})=F_{\Gamma^{\theta}}(\mathbf{w^{\prime}})&=2\sum_{e=ij\in E((K,+))}\frac{c_{b}(e)}{c_{b}(e)-1}w^{\prime}_{i}w^{\prime}_{j} \\
		&\le \frac{2r}{r-1} \sum_{e=ij\in E((K,+))}|w^{\prime}_{i}||w^{\prime}_{j}|\\
		&\le \frac{2r}{r-1}\left(\frac{r-1}{2r}\right)=1.
	\end{align*}
	
	This completes the proof of Claim \ref{Clique2}.
	
	Now, by Lemmas \ref{Yao} and \ref{Sun}, assume that $\mathbf{z}=(z_1,z_2,\ldots, z_n)^T$ is a unit non-negative eigenvector of $\lambda_{1}(\Gamma')$. Let $\mathbf{y}=(y_1,y_2,\ldots, y_n)^T$ with $y_{i}=z_{i}^{2}$, $1\le i\le n$. Then
	\begin{align*}
		\lambda_{1}(\Gamma)&=2\sum_{e=ij\in E(\Gamma^{\prime})}\sigma(ij)z_{i}z_{j}\\
		&=2\sum_{e=ij\in E^{+}(\Gamma^{\prime})}\sqrt{y_{i}y_{j}}-2\sum_{e=ij\in E^{-}(\Gamma^{\prime})}\sqrt{y_{i}y_{j}}\\
		&\le 2\sum_{e=ij\in E^{+}(\Gamma^{\prime})}\sqrt{y_{i}y_{j}}\\
		&=2\sum_{e=ij\in E^{+}(\Gamma^{\prime})}\sqrt{\frac{c_{b}(e)-1}{c_{b}(e)}}\sqrt{\frac{c_{b}(e)}{c_{b}(e)-1}y_{i}y_{j}}.
	\end{align*}
	By Cauchy-Schwarz inequality, we have
	$$
	\lambda_{1}^{2}(\Gamma)\le 4\sum_{e\in  E^{+}(\Gamma^{\prime})}\frac{c_{b}(e)-1}{c_{b}(e)} \sum_{e=ij\in  E^{+}(\Gamma^{\prime})}\frac{c_{b}(e)}{c_{b}(e)-1}y_{i}y_{j}.
	$$
	By Claim \ref{Clique2}, we obtain
	$$
	\lambda_{1}^{2}(\Gamma) \le  2\sum_{e\in  E^{+}(\Gamma^{\prime})}\frac{c_{b}(e)-1}{c_{b}(e)}.
	$$
	
	Finally, we characterize the graphs attaining the upper bounds in \eqref{1}. By the equality in \eqref{LB-JCTB-2026}, we only need to prove that the equality in \eqref{1} does not hold for any unbalanced signed graph. To the contrary, we assume that there exists an unbalanced signed graph $\Gamma$ with $\lambda_{1}(\Gamma)$ and $\omega_{b}(\Gamma)=\omega_{b}$ such that
	$$
	\lambda_{1}(\Gamma)=\sqrt{2\sum_{e\in E^{+}(\Gamma^{\prime})}\frac{c_{b}(e)-1}{c_{b}(e)}}.
	$$
	Then $E^{-}(\Gamma')\ne \emptyset$. Let $S$ be the set of underlying edges corresponding to $E^{-}(\Gamma')$. So $\Gamma'-S$ is balanced and then $\lambda_1(\Gamma'-S)$ has a non-negative eigenvector. By \eqref{1}, $\lambda_{1}(\Gamma)=\lambda_{1}(\Gamma')$, and $\lambda_{1}(\Gamma-S)=\lambda_{1}(\Gamma'-S)$, we have
	$$
	\lambda_{1}^{2}(\Gamma)\le \lambda_{1}^{2}(\Gamma-S)\le 2\sum_{e\in E^{+}(\Gamma^{\prime}-S)}\frac{c_{b}(e)-1}{c_{b}(e)}\le 2\sum_{e\in E^{+}(\Gamma^{\prime})}\frac{c_{b}(e)-1}{c_{b}(e)}.
	$$
	This implies that
	$$
	\lambda_{1}(\Gamma)=\lambda_{1}(\Gamma-S).
	$$
	
	Note that $\Gamma-S$ is balanced, since $\Gamma'-S$ is balanced. By Lemma \ref{LanCor-S}, we have $\lambda_{1}(\Gamma)<\lambda_{1}(\Gamma-S)$, a contradiction.
	
	This completes the proof.
	\qed
\end{Tproof}

\begin{remark}\label{rem1-1-1}
	{\em It is known \cite[Lemma~3.1]{Kan} that if $\Gamma$ be a signed graph with $n$ vertices, $m$ edges and frustration index $\epsilon(\Gamma)$, and $\Gamma \sim \Gamma^{\prime}$, then $|E^{+}(\Gamma^{\prime})|\le m-\epsilon(\Gamma)$. By Theorem \ref{AD1} and $c_{b}(e)\le \omega_{b}(\Gamma)$, we get $$
		\lambda_{1}^{2}(\Gamma)\le 2\sum_{e\in  E^{+}(\Gamma^{\prime})}\frac{\omega_{b}(\Gamma)-1}{\omega_{b}(\Gamma)} =2|E^{+}(\Gamma^{\prime})|\frac{\omega_{b}(\Gamma)-1}{\omega_{b}(\Gamma)} \le 2(m-\epsilon(\Gamma))\frac{\omega_{b}(\Gamma)-1}{\omega_{b}(\Gamma)}.
		$$
		Hence the bound of Theorem \ref{AD1} is better than that of \cite[Theorem~3.3]{Kan}.
	}
\end{remark}

\section{Proof of Theorem \ref{AD2}}\label{Sec-Conj-Hou-44}

\begin{lemma}\emph{(See \cite[Theorem 4.2]{Kan})}\label{Kan}
	Let $\Gamma$ be a signed graph with $n$ vertices and $m$ edges with $n \ge 3$, and $\Gamma$ has no balanced triangles. If $\lambda_{1}(\Gamma)\ge \lambda_{2}(\Gamma)\ge\cdots \ge \lambda_{n}(\Gamma)$ are eigenvalues of $A(\Gamma)$ and $\lambda_{1}(\Gamma) \ge |\lambda_{n}(\Gamma)|$, then
	$$
	\lambda_{1}(\Gamma)^{2}+\lambda_{2}(\Gamma)^{2} \le m.
	$$
\end{lemma}


\begin{lemma}\emph{(See \cite[Theorem 2.1]{Bo})}\label{Bo3}
	Let $\mathbf{x}=(x_{1}, x_{2},\dots, x_{n})^T $ and $\mathbf{y}=(y_{1}, y_{2},\dots, y_{n})^T$, where $x_i, y_i~(i=1,2,\ldots,n)$ are non-negative real numbers. If $\left \{x_{i}\right\} _{i=1}^{n}$ and $\left \{y_{i}\right\}_{i=1}^{n}$ are in non-increasing order, and
	$$\sum_{i=1}^k x_i\ge \sum_{i=1}^k y_i, ~~k=1,2,\ldots,n,$$
	then $||\mathbf{x}||_{p} \ge ||\mathbf{y}||_{p}$ for every real number $p> 1$, where equality holds if and only if $\mathbf{x}=\mathbf{y}$.
\end{lemma}

A walk $v_1v_2\dots v_k$ ($k\ge 2$) in a graph $G$ is called an \emph{internal path} if these $k$ vertices are distinct (except possibly $v_1=v_k$), $d(v_1)\ge 3, d(v_k)\ge 3$ and $d(v_2)=\dots =d(v_{k-1})=2$ (unless $k=2$). Let $G_{uv}$ denote the graph obtained by subdividing the edge $uv$ of $G$. Clearly, subdividing an edge introduces a new vertex on that edge. Let $Y_n$ be the graph obtained from an induced path $v_1v_2\cdots v_{k-4}$ by attaching two pendant vertices to each of $v_1$ and $v_{k-4}$.

\begin{lemma}\emph{(See \cite{HoffmanS75})}\label{Bo4}
	Let $G$ be a connected graph of order $n$ with $uv\in E(G)$. If $uv$ belongs to an internal path of $G$ and $G\not\cong Y_{n}$, then $\lambda_{1}(G_{uv})<\lambda_{1}(G)$.
\end{lemma}

Let $C^{(r)}_{n}$, $0\le r\le n-1$, denote a signed cycle of order $n$ and size $n$ with $r$ negative edges, where the underlying graph is the cycle $C_{n}$. For an integer $r$,  define $$
[r]=\left\{\begin{array}{ll}
	0, & \text{ if $r$ is even, } \\[0.2cm]
	1, &  \text{ if $r$ is odd. }
\end{array}
\right.
$$

\begin{lemma}\emph{(See \cite{Ger})}\label{Ger}
	Eigenvalues of $C^{(r)}_{n}$ are given by $$
	\lambda_{j}(C^{(r)}_{s})=2\cos\left(\frac{(2j-[r])\pi}{n}\right), \text{~for~}  j=1,2, \dots, n.
	$$
\end{lemma}

\begin{lemma}\emph{(See \cite{Cvet})}\label{Cvet}
	Let $\Gamma$ be a signed graph and $U$ a subset of $V(\Gamma)$ with $|U|=k$. Then
	$$
	\lambda_{i}(\Gamma) \ge \lambda_{i}(\Gamma[V(\Gamma)\setminus U]) \ge \lambda_{i+k}(\Gamma),~~\text{for $1 \le i \le n-k$}.
	$$
\end{lemma}


\begin{Tproof}\textbf{~of~Theorem~\ref{AD2}.}~Recall \cite[Theorem~1.3]{Bo} that the result is  true if $\Gamma$ is balanced.
	
	In the following, we prove that the result is also true if $\Gamma$ is unbalanced. By contradiction, we assume that $\Gamma$ contains no balanced triangle.  Consider the following two cases.
	
	\noindent\emph{Case~1.}~Assume that $\lambda_{2}(\Gamma) \ge 1$. Then $\lambda_{1}(\Gamma)^{2}+\lambda_{2}(\Gamma)^{2} \ge m$. Note that $\Gamma$ contains no balanced triangle. By Lemma \ref{Kan}, we have
	$$
	\lambda_{1}(\Gamma)^{2}+\lambda_{2}(\Gamma)^{2}=m.
	$$
	Let $n_{\Gamma}^{+}$, $n_{\Gamma}^{-}$ and $n_{\Gamma}^{0}$ denote the numbers (with multiplicities) of positive, negative and zero eigenvalues of $A(\Gamma)$, respectively. Let
	$$ s^{+}=\lambda_{1}(\Gamma)^{2}+\lambda_{2}(\Gamma)^{2}+\dots+ \lambda_{n_{\Gamma}^{+}}(\Gamma)^{2}
	$$
	and $$s^{-}=\lambda_{n-n_{\Gamma}^{-}+1}(\Gamma)^{2}+ \lambda_{n-n_{\Gamma}^{-}+2}(\Gamma)^{2}+\dots+\lambda_{n}(\Gamma)^{2}.
	$$
	Then
	$$
	\lambda_{1}(\Gamma)^{2}+\lambda_{2}(\Gamma)^{2}=m=\frac{s^{+}+s^{-}}{2}.
	$$
	Thus, we have
	$$\lambda_{1}(\Gamma)^{2}+\lambda_{2}(\Gamma)^{2} \ge 2(\lambda_{1}(\Gamma)^{2}+\lambda_{2}(\Gamma)^{2})-s^{+} =  s^{-} \ge 0.
	$$
	
	Set $\mathbf{x}=(\lambda_{1}(\Gamma)^2, \lambda_{2}(\Gamma)^2, 0,\dots, 0)^T$ and $\mathbf{y}=(\lambda_{n}(\Gamma)^2, \lambda_{n-1}(\Gamma)^2, \dots, \lambda_{n-n_{\Gamma}^{-}+1}(\Gamma))^2$. By Lemma \ref{Bo3}, we have
	$$
	||\mathbf{x}||_{\frac{3}{2}}^{\frac{3}{2}} \ge ||\mathbf{y}||_{\frac{3}{2}}^{\frac{3}{2}}.
	$$
	It implies that
	$$
	\lambda_{1}(\Gamma)^{3}+\lambda_{2}(\Gamma)^{3} \ge -\left(\lambda_{n}(\Gamma)^{3}+\lambda_{n-1}(\Gamma)^{3}+\dots +\lambda_{n-n_{\Gamma}^{-}+1}(\Gamma)^{3}\right).
	$$
	
	Let $t_{\Gamma}^{+}$ (respectively, $t_{\Gamma}^{-}$) denote the number of balanced (respectively, unbalanced) triangles. Recall that $\Gamma$ contains no balanced triangle. Then $t_{\Gamma}^{+}=0$. So,
	$$
	6(t_{\Gamma}^{+}-t_{\Gamma}^{-})\ge \lambda_{1}(\Gamma)^{3}+\lambda_{2}(\Gamma)^{3}+ \lambda_{n}(\Gamma)^{3}+\lambda_{n-1}(\Gamma)^{3}+ \dots+\lambda_{n-n_{\Gamma}^{-}+1}(\Gamma)^{3}\ge 0.
	$$
	This indicates that $t_{\Gamma}^{-}=0$. Hence $\Gamma$ contains no triangle, that is, $\Gamma$ is triangle-free.
	
	By Lemma \ref{Sun}, let $\Gamma^{\prime}$ be a signed graph such that $\Gamma^{\prime} \sim \Gamma$ and $\lambda_{1}(\Gamma^{\prime})$ has a non-negative eigenvector. Let $S$ be the set of underlying edges corresponding to $E^{-}(\Gamma')$. Then $\Gamma'-S$ is balanced. By Lemma \ref{LanCor-S},
	\begin{equation}\label{2}
		\lambda_{1}(\Gamma'-S)>\lambda_{1}(\Gamma^{\prime}) =\lambda_{1}(\Gamma) \ge \sqrt{m-1} \ge \sqrt{m_{1}},
	\end{equation}
	where $m_{1}$ is the size of $\Gamma'-S$. Since $\Gamma$ is triangle-free, $\Gamma'-S$ is triangle-free. By \cite[Theorem~2.1]{Nik2}, $\Gamma'-S$ is a complete bipartite graph. This implies that $\Gamma$ contains at least one triangle, a contradiction.


	\noindent\emph{Case~2.}~Assume that $\lambda_{2}(\Gamma)<1$. This implies that if $\Gamma$ is disconnected then its every component is an isolated vertex except for one component.
	

	We assume that $\Gamma$ is connected. Note that the adjacency eigenvalues of an unbalanced triangle are $ 1,  1,  2 $.  If $\Gamma$ has an unbalanced triangle, by Lemma \ref{Cvet}, $\lambda_{2}(\Gamma) \ge 1$, a contradiction. So, $\Gamma$ is triangle-free.

	Let $s$ be the length of a shortest odd cycle of $\Gamma$. Then $s \ge 5$.
	If $s \ge 6$, by Lemma \ref{Ger}, $\lambda_{2}(C^{(r)}_{s}) \ge 1$. By Lemma \ref{Cvet}, $\lambda_{2}(\Gamma) \ge \lambda_{2}(C^{(r)}_{s}) \ge 1$, a contradiction. Thus $s=5$.
	
	\setcounter{claim}{0}
	\begin{claim}\label{triangle-1}
		$\Gamma$ has no unbalanced $C_{5}$.
	\end{claim}
	
	\noindent\emph{Proof of Claim 1.}~If $\Gamma$ has an unbalanced $(C_{5},\sigma_{C_{5}})$, by Lemmas \ref{Ger} and \ref{Cvet}, we have
	$$\lambda_{2}(\Gamma)\ge \lambda_{2}((C_{5},\sigma_{C_{5}}))=2\cos\left(\frac{\pi}{5}\right)>1,$$
	a contradiction.

	\begin{claim}\label{triangle-2}
		$\Gamma$ does not contain $(H_{i},\sigma_{i})$ as an induced subgraph, that is, $\Gamma$ is $(H_{i},\sigma_{i})$-free, where $H_{i} ~(i=1,2,3)$ are shown in Figure \ref{figure1}.
	\end{claim}
	
	\noindent\emph{Proof of Claim 2.}~If $(H_{i},\sigma_{i})$ is balanced, then $\lambda_{2}((H_{i},+))=1$. By Lemma \ref{Cvet}, this contradicts that $\lambda_{2}(\Gamma)<1$. By Claim \ref{triangle-1}, we know that $(H_{1},\sigma_{1})$ is balanced. So, $\Gamma$ is $(H_{1},\sigma_{1})$-free. If $(H_{2},\sigma_{2})$ is unbalanced, by equivalent switching, $(H_{2},\sigma_{2})$ has only two cases, as shown in Figure \ref{figure2}. By a simple computation, $\lambda_{2}((H_{2},\sigma_{2})) \approx 1.629$ or $1.732$. By Lemma \ref{Cvet}, this contradicts that $\lambda_{2}(\Gamma)< 1$.  So, $\Gamma$ is $(H_{2},\sigma_{2})$-free. Similarly, if $(H_{3},\sigma_{3})$ is unbalanced, then $(H_{3},\sigma_{3})$ has only two cases, as shown in Figure \ref{figure3}, which also contradicts that $\lambda_{2}(\Gamma)< 1$. So, $\Gamma$ is $(H_{3},\sigma_{3})$-free.
	
	\begin{figure}[H]
		\begin{minipage}[c]{0.3\textwidth}
			\hspace{0.5cm}
			\begin{tikzpicture}[scale =1]
				\node (u1) at (0.5,2) {$u_{1}$};
				\node (u2) at (-0.5,1) {$u_{2}$};
				\node (u3) at (-0.5,0) {$u_{3}$};
				\node (u4) at (1.5,0) {$u_{4}$};
				\node (u5) at (1.5,1) {$u_{5}$};
				\node (v)  at (3,2) {$v$};
				\draw [line width=1.5pt](u1) -- (u2) -- (u3) -- (u4) -- (u5) -- (u1);
				\draw [line width=1.5pt](v) -- (u1);
				\node (H_1)  at (0.5,-1) {$H_1$};
			\end{tikzpicture}
		\end{minipage}
		\hspace{0.02\textwidth}
		\begin{minipage}[c]{0.3\textwidth}
			\begin{tikzpicture}[scale =1]
				\node (u1) at (0.5,2) {$u_{1}$};
				\node (u2) at (-0.5,1) {$u_{2}$};
				\node (u3) at (-0.5,0) {$u_{3}$};
				\node (u4) at (1.5,0) {$u_{4}$};
				\node (u5) at (1.5,1) {$u_{5}$};
				\node (v)  at (3,2) {\(v\)};
				\node (w)  at (3,-0.5) {\(w\)};
				\draw [line width=1.5pt](u1) -- (u2) -- (u3) -- (u4) -- (u5) -- (u1);
				\draw [line width=1.5pt](v) -- (u1);
				\draw [line width=1.5pt](v) -- (u4);
				\draw [line width=1.5pt](w) -- (u3);
				\draw [line width=1.5pt](w) -- (u5);
				\node (H_2)  at (1,-1) {$H_2$};
			\end{tikzpicture}
		\end{minipage}
		\hspace{0.02\textwidth}
		\begin{minipage}[c]{0.08\textwidth}
			\begin{tikzpicture}[scale =1]
				\node (u1) at (0.5,2) {$u_{1}$};
				\node (u2) at (-0.5,1) {$u_{2}$};
				\node (u3) at (-0.5,0) {$u_{3}$};
				\node (u4) at (1.5,0) {$u_{4}$};
				\node (u5) at (1.5,1) {$u_{5}$};
				\node (v)  at (2.5,2) {\(v\)};
				\node (w)  at (-1.5,2) {\(w\)};
				\draw [line width=1.5pt](u1) -- (u2) -- (u3) -- (u4) -- (u5) -- (u1);
				\draw [line width=1.5pt](v) -- (u1);
				\draw [line width=1.5pt](v) -- (u4);
				\draw [line width=1.5pt](w) -- (u3);
				\draw [line width=1.5pt](w) -- (u1);
				\node (H_3)  at (0.5,-1) {$H_3$};
			\end{tikzpicture}
		\end{minipage}
		\caption{Three graphs $ H_{1} $, $ H_{2} $ and $ H_{3} $. } 
		\label{figure1} 
\end{figure}
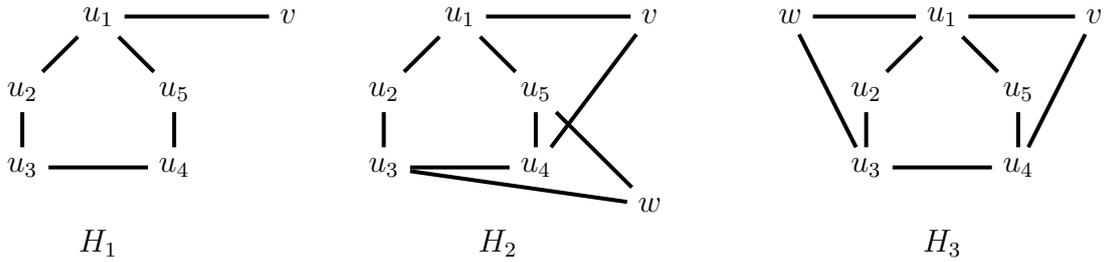

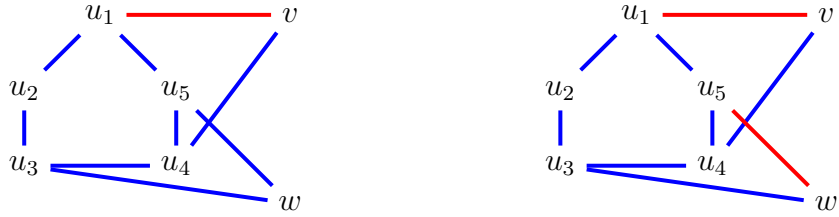
\begin{figure}[H]
	\begin{minipage}[c]{0,6\textwidth}
		\hspace{2.5cm}
		\begin{tikzpicture}[scale =1]
			\node (u1) at (0.5,2) {$u_{1}$};
			\node (u2) at (-0.5,1) {$u_{2}$};
			\node (u3) at (-0.5,0) {$u_{3}$};
			\node (u4) at (1.5,0) {$u_{4}$};
			\node (u5) at (1.5,1) {$u_{5}$};
			\node (v)  at (3,2) {\(v\)};
			\node (w)  at (3,-0.5) {\(w\)};
			\draw [line width=1.5pt][blue](u1) -- (u2) -- (u3) -- (u4) -- (u5) -- (u1);
			\draw [line width=1.5pt][red](v) -- (u1);
			\draw [line width=1.5pt][blue](v) -- (u4);
			\draw [line width=1.5pt][blue](w) -- (u3);
			\draw [line width=1.5pt][blue](w) -- (u5);
		\end{tikzpicture}
	\end{minipage}
	\hfill
	\begin{minipage}[c]{1\textwidth}
		\begin{tikzpicture}[scale =1]
			\node (u1) at (0.5,2) {$u_{1}$};
			\node (u2) at (-0.5,1) {$u_{2}$};
			\node (u3) at (-0.5,0) {$u_{3}$};
			\node (u4) at (1.5,0) {$u_{4}$};
			\node (u5) at (1.5,1) {$u_{5}$};
			\node (v)  at (3,2) {\(v\)};
			\node (w)  at (3,-0.5) {\(w\)};
			\draw [line width=1.5pt][blue](u1) -- (u2) -- (u3) -- (u4) -- (u5) -- (u1);
			\draw [line width=1.5pt][red](v) -- (u1);
			\draw [line width=1.5pt][blue](v) -- (u4);
			\draw [line width=1.5pt][blue](w) -- (u3);
			\draw [line width=1.5pt][red](w) -- (u5);
		\end{tikzpicture}	
		\hspace{2cm}
	\end{minipage}
	\caption{Unbalanced $(H_{2},\sigma_{2})$'s, where blue edges represent positive edges and red edges represent negative edges.} 
	\label{figure2} 
	\end{figure}
	
	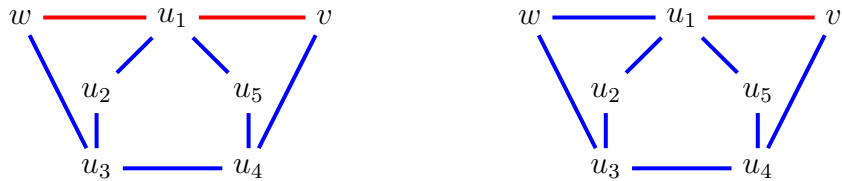
\begin{figure}[H]
\centering
\begin{minipage}[c]{0,6\textwidth}
	\hspace{2.5cm}
	\begin{tikzpicture}[scale =1]
		\node (u1) at (0.5,2) {$u_{1}$};
		\node (u2) at (-0.5,1) {$u_{2}$};
		\node (u3) at (-0.5,0) {$u_{3}$};
		\node (u4) at (1.5,0) {$u_{4}$};
		\node (u5) at (1.5,1) {$u_{5}$};
		\node (v)  at (2.5,2) {\(v\)};
		\node (w)  at (-1.5,2) {\(w\)};
		\draw [line width=1.5pt][blue](u1) -- (u2) -- (u3) -- (u4) -- (u5) -- (u1);
		\draw [line width=1.5pt][red](v) -- (u1);
		\draw [line width=1.5pt][blue](v) -- (u4);
		\draw [line width=1.5pt][blue](w) -- (u3);
		\draw [line width=1.5pt][red](w) -- (u1);
	\end{tikzpicture}
\end{minipage}
\hfill
\begin{minipage}[c]{0.3\textwidth}
	\hspace{-2cm}
	\begin{tikzpicture}[scale =1]
		\node (u1) at (0.5,2) {$u_{1}$};
		\node (u2) at (-0.5,1) {$u_{2}$};
		\node (u3) at (-0.5,0) {$u_{3}$};
		\node (u4) at (1.5,0) {$u_{4}$};
		\node (u5) at (1.5,1) {$u_{5}$};
		\node (v)  at (2.5,2) {\(v\)};
		\node (w)  at (-1.5,2) {\(w\)};
		\draw [line width=1.5pt][blue](u1) -- (u2) -- (u3) -- (u4) -- (u5) -- (u1);
		\draw [line width=1.5pt][red](v) -- (u1);
		\draw [line width=1.5pt][blue](v) -- (u4);
		\draw [line width=1.5pt][blue](w) -- (u3);
		\draw [line width=1.5pt][blue](w) -- (u1);
	\end{tikzpicture}
\end{minipage}
\caption{Unbalanced $(H_{3},\sigma_{3})$'s, where blue edges represent positive edges and red edges represent negative edges. } 
\label{figure3} 
\end{figure}

Define
$$
N(U)=\cup_{v\in U}N_{\Gamma}(v).
$$
Let
$$
T_1=\left \{u_{i}: 1\le i \le 5 \right \}
$$
denote the vertex set of $(C_5,\sigma)$ in $\Gamma$ with  $\Gamma[T_1]=u_{1}u_{2}u_{3}u_{4}u_{5}u_{1}$, and
$$
T_2=N(T_1)\setminus T_1.
$$

\begin{claim}\label{triangle-3}
$d_{T_1}(v)=2$ for each $v \in T_2$.
\end{claim}

\noindent\emph{Proof of Claim 3.}~For $v\in T_2$, without loss of generality, assume that $v\in N_\Gamma(u_{1})$. If $d_{T_1}(v)\ge 3$, then $\Gamma$ contains a triangle, a contradiction. If $d_{T_1}(v)=1$, then $N_{T_1}(v)=u_{1}$. So, $\left \{v,u_{1},u_{2},u_{3},u_{4},u_{5}  \right \}$ induces $(H_{1},\sigma_{1})$, a contradiction to Claim \ref{triangle-2}. Thus, $d_{T_1}(v)=2$.

\begin{claim}\label{triangle-4}
$V(\Gamma)=T_1\cup T_2$.
\end{claim}

\noindent\emph{Proof of Claim 4.}~By contradiction, let $v^{\prime}\in V(\Gamma)\setminus(T_1\cup T_2)\not=\emptyset$. Then $v^{\prime}u_{i}\notin E(\Gamma)$ for $1\le i \le 5$. Without loss of generality, assume that $v\in T_2$ and $u_1v, v^{\prime}v\in E(\Gamma)$. Recall that $d_{T_1}(v)=2$ and $\Gamma$ is triangle-free. Without loss of generality, assume that $N_{T_1}(v)=\left \{u_{1},u_{3}\right \}$. So, $\left \{v^{\prime},v,u_{3},u_{4},u_{5},u_{1} \right \}$ induces $(H_{1},\sigma_{1})$, a contradiction. Therefore, $V(\Gamma)=T_1\cup T_2$.

\begin{claim}\label{triangle-5}
$n\ge 7$.
\end{claim}
\noindent\emph{Proof of Claim 5.}~If $n=5$, by Claim \ref{triangle-1}, $\Gamma \sim (C_{5},+)$, which contradicts that $\Gamma$ is unbalanced. If $n=6$, by Claims \ref{triangle-3} and \ref{triangle-4}, we have $m=7$. Without loss of generality, assume that $v\in T_2$ and $N_{T_1}(v)=\left \{u_{1},u_{4} \right \}$. By Claim \ref{triangle-1}, there are only two cases for $\Gamma$, as shown in Figure \ref{figure4}. Note that $\Gamma_1$ is balanced, a contradiction to that $\Gamma$ is unbalanced. For $\Gamma_2$, by a simple computation, we have $\lambda_{1}(\Gamma_2)\approx 2.391$. Clearly, $\lambda_{1}(\Gamma_2)< \sqrt{m-1}=\sqrt{6}$, a contradiction. Therefore, $n\ge 7$.

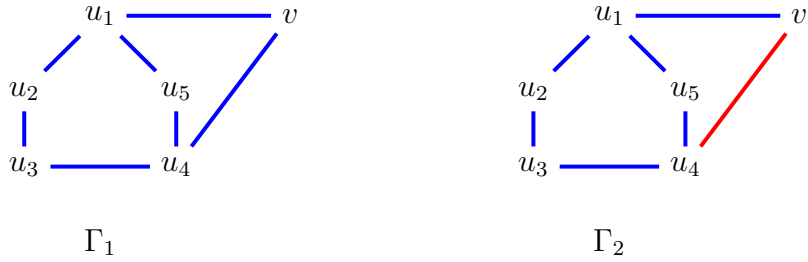
\begin{figure}[H]
\begin{minipage}[c]{0,6\textwidth}
\hspace{2.5cm}
\begin{tikzpicture}[scale =1]
	\node (u1) at (0.5,2) {$u_{1}$};
	\node (u2) at (-0.5,1) {$u_{2}$};
	\node (u3) at (-0.5,0) {$u_{3}$};
	\node (u4) at (1.5,0) {$u_{4}$};
	\node (u5) at (1.5,1) {$u_{5}$};
	\node (v)  at (3,2) {$v$};
	\draw [line width=1.5pt][blue](u1) -- (u2) -- (u3) -- (u4) -- (u5) -- (u1);
	\draw [line width=1.5pt][blue](v) -- (u1);
	\draw [line width=1.5pt][blue](v) -- (u4);
	\node (Gamma_1)  at (0.5,-1) {$\Gamma_1$};
\end{tikzpicture}
\end{minipage}
\hfill
\begin{minipage}[c]{0.3\textwidth}
\hspace{-2cm}
\begin{tikzpicture}[scale =1]
	\node (u1) at (0.5,2) {$u_{1}$};
	\node (u2) at (-0.5,1) {$u_{2}$};
	\node (u3) at (-0.5,0) {$u_{3}$};
	\node (u4) at (1.5,0) {$u_{4}$};
	\node (u5) at (1.5,1) {$u_{5}$};
	\node (v)  at (3,2) {$v$};
	\draw [line width=1.5pt][blue](u1) -- (u2) -- (u3) -- (u4) -- (u5) -- (u1);
	\draw [line width=1.5pt][blue](v) -- (u1);
	\draw [line width=1.5pt][red](v) -- (u4);
	\node (Gamma_2)  at (0.5,-1) {$\Gamma_2$};
\end{tikzpicture}
\end{minipage}
\caption{$\Gamma$'s with $n=6$ and $m=7$, where blue edges represent positive edges and red edges represent negative edges.} 
\label{figure4} 
\end{figure}

Let $v \in T_2$ and $N_{T_1}(v)=\left \{u_{1},u_{4} \right \}$. Recall that $\Gamma$ is triangle-free. For any $w\in T_2\setminus\{v\}$, if $N_{T_1}(w)=N_{T_1}(v)$, then $wv\notin E(\Gamma)$; if $N_{T_1}(w)\ne N_{T_1}(v)$,  then $N_{T_1}(w)\cap N_{T_1}(v)=\emptyset$, since $\Gamma$ is $(H_{3},\sigma_{3})$-free. Define
$$A=\{w\in T_2  \mid N_{T_1}(w)=N_{T_1}(v)\}$$
and
$$B=\{w\in T_2 \mid N_{T_1}(w)\cap N_{T_1}(v)=\emptyset\}.$$
Recall that $\Gamma$ is $(H_{2},\sigma_{2})$-free.  If $w\in B$, then $wv\in E(\Gamma)$. Recall that $\Gamma$ is $(H_{3},\sigma_{3})$-free. Then $\cup_{w\in B}N_{T_1}(w)=\left \{u_{2},u_{5}\right \}$ or $\cup_{w\in B}N_{T_1}(w)=\left \{u_{3},u_{5}\right \}$.



Without loss of generality,  assume that $\cup_{w\in B}N_{T_1}(w)=\left \{u_{2},u_{5}\right \}$. Let $C=N_{\Gamma}(u_{1})\cap N_{\Gamma}(u_{4})$ and $D=N_{\Gamma}(u_{2})\cap N_{\Gamma}(u_{5})$.
Clearly, $C$ and $D$ are both independent sets and $C\cup D=T_2\cup \left \{u_{1},u_{5}\right \}$. Therefore, $\Gamma[C\cup D]$ is signed complete bipartite subgraph. Let $|C|=c$ and $|D	|=d$. Then $m=(c+1)(d+1)+1$. By Claim \ref{triangle-1}, there exists a signed graph $\Gamma^{''}=(G,\sigma^{''})$ such that $\Gamma^{''} \sim \Gamma$ and $\Gamma^{''}[T_1]=(C_{5},+)$. Let $\sigma_{K}$ be a sign function such that $(SK_{c+1,d+1},\sigma_K) \sim \Gamma^{''}$. By Lemmas \ref{Yao} and \ref{Bo4}, we have
\begin{align*}
\lambda_{1}(\Gamma)=\lambda_{1}((SK_{c+1,d+1},\sigma_{K}))
&\le\lambda_{1}((SK_{c+1,d+1},+))\\
&<\lambda_{1}((K_{c+1,d+1},+))\\
&=\sqrt{(c+1)(d+1)}\\
&=\sqrt{m-1},
\end{align*}
a contradiction.

This completes the proof.
\qed
\end{Tproof}

\begin{remark}
{\em Let $\Gamma=(K_{4},\sigma)$ be a complete graph of order $4$ having exactly one negative edge. By a simple calculation, $\lambda_{1}(\Gamma)=|\lambda_{4}(\Gamma)|=\sqrt{5}=\sqrt{|E(\Gamma)|-1}$, which satisfies the conditions of Theorem \ref{AD2}. Since $\Gamma$ is unbalanced, Theorem \ref{AD2} is a signed version of \cite[Theorem~1.3]{Bo}.
}
\end{remark}

\section{Proof of Conjecture \ref{ConjecHou-1}}\label{Sec-Conj-Hou-1}


For a square matrix $M$, we define $\det(M)$ as its determinant. A zero matrix is the matrix which all elements are zero.

\begin{lemma}\emph{(See \cite{Hou})}\label{Hou}
Let $\Gamma$ be a connected signed graph. Then $\det(L(\Gamma))=0$ if and only if $\Gamma$ is balanced.
\end{lemma}




\medskip
\begin{Tproof}\textbf{~of~Conjecture~\ref{ConjecHou-1}.}
Let $1\le k\le n-1$. Partition $L(\Gamma)$ as
$$
\begin{pmatrix}
B	& C\\
C^{T}	&D
\end{pmatrix},
$$
where $B$ is $k$-by-$k$. Consider the following two cases.

\noindent\emph{Case~1.}~If $B$ is not a positive definite matrix, then $B$ has a zero eigenvalue, since $B$ is a positive semi-definite matrix. We obtain $\det(L(\Gamma))=0$. By Lemma \ref{Hou}, we have $\Gamma$ is balanced. By Theorem \ref{Grone}, we obtain Conjecture \ref{ConjecHou-1} is true.

\noindent\emph{Case~2.}~If $B$ is a positive definite matrix, then $B$ is invertible.  Since $G$ is connected, $C$ is not zero matrix. Note that $L(\Gamma)=FF^{T}$, where
$$
F=\begin{pmatrix}
B^{\frac{1}{2} }	& 0\\
C^{T}B^{-\frac{1}{2}}	&H
\end{pmatrix}
$$
and $H=(D-C^{T}B^{-1}C)^{\frac{1}{2}}$. Note also that $L(\Gamma)$ and
$$
K=F^{T}F=\begin{pmatrix}
B+B^{-\frac{1}{2}}CC^{T}B^{-\frac{1}{2}}	& *\\
*	&*
\end{pmatrix}
$$
have the same eigenvalues.	Since $B$ is a positive definite matrix, there exists an orthogonal matrix $Q$ such that $B=Q^{T}\diag(a_{1},a_{2},\dots ,a_{n})Q$, where $a_{i}>0,~1\le i \le n$.

Notice that $QCC^{T}Q^{T}=(QC)(QC)^{T}$ is a positive semi-definite matrix. We assume that the non-negative diagonal elements are $m_{1}, m_{2}, \dots , m_{k}$. If there exists an $i$ such that $m_{i}=0,~i\in \{1,2,\dots,k\}$, then all the elements in the row and column corresponding to $m_i$ are zero \cite[P201]{Zhang}. So, if $m_{1}=m_{2}=\dots =m_{k}=0$, then $(QC)(QC)^{T}$ is zero matrix. Recall that $(QC)(QC)^{T}$ is a positive semi-definite matrix. Then $QC$ is zero matrix. Recall that $Q$ is an orthogonal matrix. Then $C$ is zero matrix, a contradiction. Hence, there exists a $m_{j}>0$ for a certain $j$. So,
\begin{align*}
\trace\left(B^{-\frac{1}{2}}CC^{T}B^{-\frac{1}{2}}\right) &=\trace\left(B^{-1}CC^{T}\right) \\
&=\trace\left(Q^{T}\diag\left(\frac{1}{a_{1}},\frac{1}{a_{2}},\dots, \frac{1}{a_{k}}\right)QCC^{T}\right)\\
&=\trace\left(\diag\left(\frac{1}{a_{1}},\frac{1}{a_{2}},\dots, \frac{1}{a_{k}}\right)QCC^{T}Q^{T}\right)\\
&=\sum_{i=1}^{k}\frac{m_{i}}{a_{i}}>0.
\end{align*}
Now,
$$
\sum_{i=1}^{k}\mu_{i}(\Gamma)=\trace\left(B+B^{-\frac{1}{2}}CC^{T}B^{-\frac{1}{2}}\right)>\trace(B) =\sum_{i=1}^{k}d_{i}(\Gamma).
$$
This completes the proof.
\qed\end{Tproof}

\section{Proof~of~Theorem~\ref{LP1}}\label{Sec-Conj-Hou-22}


Let $\Gamma=(G,\sigma)$ be a signed graph. For each edge $e_{k}=(v_{i},v_{j})$ of $\Gamma$, choose one of $v_{i}$ or $v_{j}$ to be the head of $e_{k}$ and the other to be the tail. The vertex-edge \emph{incidence matrix} $P=P(\Gamma)$ afforded by a fixed orientation of $\Gamma$ is the $n$-by-$m$ matrix $P=(c_{ij})$ given by
$$
c_{ij} =
\begin{cases}
+1, & \text{if } v_i \text{ is the head of } e_j; \\
-1, & \text{if } v_i \text{ is the tail of } e_j, \text{ and } \sigma(e_j) = +1; \\
+1, & \text{if } v_i \text{ is the tail end of } e_j, \text{ and } \sigma(e_j) = -1; \\
0,  & \text{otherwise}.
\end{cases}
$$
It is easy to verify that $PP^{T}$ is always the Laplacian matrix $L(\Gamma)$. In any event, $P^{T}P$ and $PP^{T}$ share the same nonzero eigenvalues\cite{Ping}.

\begin{lemma}\emph{(See \cite{Cvt})}\label{Cvt}
Let $M$ be a positive semi-definite matrix with eigenvalues $\lambda_1 \geq \lambda_2 \geq \cdots \geq \lambda_n$. Then, for $r = 1, 2, \ldots, n$,
$$
\lambda_1 + \lambda_2 + \cdots + \lambda_r = \max \left\{ \mathbf{u}_1^T M \mathbf{u}_1 + \mathbf{u}_2^T M \mathbf{u}_2 + \cdots + \mathbf{u}_r^T M \mathbf{u}_r \right\},
$$
where the maximum is taken over all orthonormal vectors $\mathbf{u}_1, \mathbf{u}_2, \ldots, \mathbf{u}_r$. In particular, $\lambda_1 + \lambda_2 + \cdots + \lambda_r$ is bounded below by the sum of the $r$ largest diagonal entries of $M$.
\end{lemma}

\begin{lemma}\emph{(See \cite{Wang})}
\label{Wang}
Let $\Gamma$ be a signed graph with n vertices and let $\Gamma'$ be a signed graph obtained from $\Gamma$ by inserting a new edge into $\Gamma$. Then the Laplacian eigenvalues of $\Gamma$ and $\Gamma'$ interlace, that is,
$$
\mu_{1}(\Gamma')\ge \mu_{1}(\Gamma) \ge \mu_{2}(\Gamma')\ge \mu_{2}(\Gamma)\ge \dots \ge \mu_{n}(\Gamma')\ge \mu_{n}(\Gamma).
$$	
\end{lemma}


In the following, we give the proof of Theorem \ref{LP1}. We would like to mention that the proof of Theorem \ref{LP1} is similar to that of \cite[Theorem~2]{GM-SIAM-1994}. For the sake of completeness, we keep the proof of Theorem \ref{LP1}.


\medskip
\begin{Tproof}\textbf{~of~Theorem~\ref{LP1}.}
By the previous definition, $L(\Gamma)=PP^{T}$, regardless of the orientation. The matrix $K(\Gamma)=P^{T}P$ depends on the orientation. In any event, $K(\Gamma)$ and $L(\Gamma)$ share the same nonzero eigenvalues.

Suppose that $\mathbf{x}$ is a real $m$-dimensional column vector. If $e\in E(\Gamma)$ and $v\in V(\Gamma)$, define $s(v,e)=+1$ if $v$ is the head of $e$, $s(v,e)=-1$ if $v$ is the tail of $e$ and $\sigma(e)=+1$, $s(v,e)=+1$ if $v$ is the tail of $e$ and $\sigma(e)=-1$, and $0$ otherwise. It is convenient to think of its components as indexed by $E(\Gamma)$, so the "$v$-th component" of $P\mathbf{x}$ is
$$\sum_{e\in  E(\Gamma)}s(v,e)\mathbf{x}_{e}$$
where $\mathbf{x}_{e}$ denotes the coordinate of $\mathbf{x}$ corresponding to the edge $e$. Therefore,
\begin{equation}\label{3}
\mathbf{x}^{T}K(\Gamma)\mathbf{x}=\sum_{v\in V(\Gamma)}\left(\sum_{e\in E(\Gamma)}s(v,e)\mathbf{x}_{e}\right)^2.
\end{equation}

Choose an orientation of $E(\Gamma)$ such that $u_{i}$ is the head of each of the $d_{\Gamma}(u_{i})$ edges incident with it,  $1\le i\le k-r$. In addition, we may prescribe that $u_{i}$ is the head of each of the $d_{\Gamma}(u_{i})-1>0$ edges, other than $e_{k-i+1}$, incident with it, $k-r<i\le k$.

For each $u\in R$, define a real $m$-tuple $\mathbf{x}(u)$ as follows: $\mathbf{x}(u)_{e}$, the coordinate of $\mathbf{x}(u)$ corresponding to the edge $e\in E(\Gamma)$, is 1 if $u$ is the head of $e$, and 0 otherwise. Then $\left \{\mathbf{x}(u): u\in R  \right \} $ is an orthogonal set of vectors. Moreover, $||\mathbf{x}(u)_{i}||^2$ = $d_{\Gamma}(u_{i})$ if $i\le k-r$ and $d_{\Gamma}(u_{i})-1$ if $i>k-r$. Let $\mathbf {y}(u)=\frac{\mathbf{x}(u)}{||\mathbf{x}(u)||} $ for $u\in R$.

Then, for $1\le i\le k-r$,
$$
\sum_{e\in E(\Gamma)}s(v,e)\mathbf{y}(u_{i})_{e}=
\begin{cases}
d_{\Gamma}(u_{i})^{\frac{1}{2}}, & \text{~if~}v=u_{i},\\
-\frac{\delta_{i}}{d_{\Gamma}(u_{i})^{\frac{1}{2}}},& \text{~if~}vu_{i} \in E(\Gamma),\\
0,& \text{~otherwise},
\end{cases}
$$
and for $i>k-r$,
$$
\sum_{e\in E(\Gamma)}s(v,e)\mathbf{y}(u_{i})_{e}=
\begin{cases}
(d_{\Gamma}(u_{i})-1)^{\frac{1}{2}}, & \text{~if~}v=u_{i},\\		
-\frac{\delta_{i}}{(d_{\Gamma}(u_{i})-1)^{\frac{1}{2}}}, & \text{~if~}vu_{i}\in E(\Gamma)\setminus E(\Gamma[R]),\\				0, & \text{~otherwise}.
\end{cases}
$$
So, from \eqref{3},
$$
\mathbf{y}^{T}(u_{i})K(\Gamma)\mathbf{y}(u_{i})=
\begin{cases}
d_{\Gamma}(u_{i})+1, & 1\le i\le k-r,	\\[0.2cm]
d_{\Gamma}(u_{i}), & k-r<i\le k.
\end{cases}
$$
Since  $\left \{\mathbf{y}(u_{i}): 1\le i\le k  \right \} $ is an orthogonal set of vectors, by Lemma \ref{Cvt}, we have
$$
\sum_{i=1}^{k}\mu_{i}(\Gamma)\ge \sum_{i=1}^{k}\mathbf{y}^{T}(u_{i})K(\Gamma)\mathbf{y}(u_{i}) =\sum_{i=1}^{k}d_{\Gamma}(u_{i})+k-r.
$$

This completes the proof.
\qed
\end{Tproof}

\begin{cor}\emph{(See \cite{Hou})}
If $\Gamma=(G,\sigma)$ has an edge, then $\mu_{1}(\Gamma)\ge d_{1}(\Gamma)+1$.
\end{cor}

\begin{proof}
If $d_{1}(\Gamma)=1$, then every connected component of $\Gamma$ is either an isolated vertex or a copy of $K_{2}$. In this case, $\mu_{1}(\Gamma)=2$ and $d_{1}(\Gamma)=1$. So $\mu_{1}(\Gamma)=d_{1}(\Gamma)+1$.

If $d_{1}(\Gamma)>1$, let $w$ be a vertex of $\Gamma$ of degree $d_{1}(\Gamma)$. Let $(Z,\sigma)$ be a connected component of $\Gamma$ containing $w$. By Lemma \ref{Wang} and Theorem \ref{LP1}, $\mu_{1}(\Gamma)\ge \mu_{1}((Z,\sigma))\ge d_{1}(\Gamma)+1$.
\qed\end{proof}

\begin{cor}
Let $\Gamma=(G,\sigma)$ be a connected signed graph with $n>2$ vertices. Then
$$
\mu_{1}(\Gamma)+\mu_{2}(\Gamma)\ge d_{1}(\Gamma)+d_{2}(\Gamma)+1.
$$
If there are two nonadjacent vertices in $\Gamma$ having degree $d_{1}(\Gamma)$ and degree $d_{2}(\Gamma)$, then
$$
\mu_{1}(\Gamma)+\mu_{2}(\Gamma)\ge d_{1}(\Gamma)+d_{2}(\Gamma)+2.
$$
\end{cor}

\begin{proof}
The proof is immediate from Theorem \ref{LP1}. Either $r = 1$ or $r = 0$.
\qed
\end{proof}


\section*{Statements and Declarations}

\noindent \textbf{Competing interests}~~No potential competing interest was reported by the authors.

\medskip

\noindent \textbf{Data availability statements}~~Data sharing not applicable to this article as no datasets were generated or analysed during the current study.

\end{document}